\documentclass[USenglish]{book}

\usepackage[utf8]{inputenc}
\usepackage[small]{dgruyter}
\usepackage{amsmath,amsfonts,amsthm,amssymb,eucal,enumerate,graphicx,color}

\newtheorem{thm}{Theorem}

\newtheorem{cor}{Corollary}
\newtheorem{lem}{Lemma}

\theoremstyle{definition}
\newtheorem{defn}{Definition}
\newtheorem{rem}{Remark}

\newcommand{\R}{\mathbb{R}}

\newcommand{\N}{\mathbb{N}}
\renewcommand{\S}{\mathcal{S}}

\newcommand{\dive}{\operatorname{div}}

\newcommand{\PV}{\operatorname{P.V.}}

\newcommand{\vep}{\varepsilon}

\newcommand{\K}{\mathcal{K}}

\allowdisplaybreaks

\contribution

\begin{document}

\contributionauthor{Pablo Ra\'ul Stinga}
\affil{Department of Mathematics, Iowa State University, 396 Carver Hall, Ames, IA 50011,
United States of America, e-mail: stinga@iastate.edu}
\runningauthor{Pablo Ra\'ul Stinga}
\contributiontitle{User's guide to the fractional Laplacian and the method of semigroups}
\runningtitle{User's guide}

\abstract{The \textit{method of semigroups} is a unifying, widely applicable,
general technique to formulate and analyze fundamental aspects of fractional powers of operators $L$
and their regularity properties in related functional spaces.
The approach was introduced by the author and Jos\'e L.~Torrea in 2009 (arXiv:0910.2569v1).
The aim of this chapter is to show how the method works in the particular case of
the fractional Laplacian $L^s=(-\Delta)^s$, $0<s<1$.
The starting point is the semigroup formula for the fractional Laplacian.
From here, the classical heat kernel permits us to obtain the pointwise formula for $(-\Delta)^su(x)$.
One of the key advantages is that our technique relies on the use of heat kernels, which allows for applications
in settings where the Fourier transform is not the most suitable tool.
In addition, it provides explicit constants that are
key to prove, under minimal conditions on $u$, the validity of the pointwise limits
$$\lim_{s\to1^-}(-\Delta)^su(x)=-\Delta u(x)\quad\hbox{and}\quad\lim_{s\to0^+}(-\Delta)^su(x)=u(x).$$
The formula for the solution to the Poisson problem
$(-\Delta)^su=f$ is found through the semigroup approach as the inverse of the fractional Laplacian
$u(x)=(-\Delta)^{-s}f(x)$ (fundamental solution).
We then present the Caffarelli--Silvestre extension problem, whose
explicit solution is given by the semigroup formulas
that were first discovered by the author and Torrea. With the extension technique, an interior Harnack
inequality and derivative estimates for fractional harmonic functions can be obtained.
The classical H\"older and Schauder estimates
$$(-\Delta)^{\pm s}:C^\alpha\to C^{\alpha\mp 2s}$$
are proved with the method of semigroups in a rather quick, elegant way.
The crucial point for this will be the characterization of H\"older and Zygmund spaces
with heat semigroups.}

\keywords{method of semigroups; fractional Laplacian; extension problem; regularity estimates}
\classification{35R11; 26A33; 58J35}

\makecontributiontitle
\DOI{}

\newpage

\section{Introduction}

Fractional powers, both positive and negative, as well as complex, of linear operators appear in many areas of mathematics.
In particular, the fractional powers of the Laplacian are nowadays classical objects.
Fractional operators appear in
potential theory \cite{Berg,Bogdan,Landkof,Song-Vondracek},
probability \cite{Applebaum, Bass-Levin, Bertoin, Bochner, Bochner-book, BB, Botcher, Courrege, Molchanov-Ostrovskii},
fractional calculus and hypersingular integrals \cite{Gatto-Segovia-Vagi, Samko-hyper, Samko,Schilling,SSV},
harmonic analysis \cite{Auscher, Bochner-book,Gatto-Segovia-Vagi, Seeley, Stein-rojo,Stein},
functional analysis \cite{Balakrishnan,Grubb-regularity,Kato,Komatsu,Lions-Magenes,Yosida},
and pseudo-differential operators \cite{Eidelman,Grubb-Hormander,Jacob1,Jacob2,Jacob3, JS}.

In recent years, the fractional Laplacian or, more generally, nonlocal equations of fractional order,
gained a lot of attention from the partial differential equations research community.
It can be said that the main driving force for this has been the fundamental work of Luis A.~Caffarelli
and his collaborators, see
\cite{Caffarelli-Salsa-Silvestre, Caffarelli-Silvestre CPDE, Caffarelli-Silvestre-ARMA,
Caffarelli-Silvestre-CPAM, Caffarelli-Silvestre-Annals, Caffarelli-Vasseur, Silvestre Thesis, Silvestre CPAM}, just to mention a few.

To introduce the notion of fractional Laplacian, let $u$ be a function in the Schwartz class $\mathcal{S}=\S(\R^n)$, $n\geq1$.
The Fourier transform of $u$, denoted by $\widehat{u}$, is also in $\S$. For the Laplacian $-\Delta$ on $\R^n$ we have
$$\widehat{(-\Delta)u}(\xi)=|\xi|^2\widehat{u}(\xi)\quad\hbox{for every}~\xi\in\R^n.$$
The fractional Laplacian $(-\Delta)^s$, $0<s<1$, is then defined in a natural way as
\begin{equation}\label{eqFractionalLaplacian:Fouriertransformdefinition}
\widehat{(-\Delta)^su}(\xi)=|\xi|^{2s}\widehat{u}(\xi).
\end{equation}

\subsection{A few applications}

Let us begin by briefly describing problems in probability, financial mathematics, elasticity
and biology where fractional powers of differential operators appear. 

\vskip 0.2cm
\noindent\texttt{I. L\'evy processes.} Let $(X_t:t\geq0)$ be a symmetric $2s$-stable ($0<2s\leq2$) $\R^n$-valued
L\'evy process starting at $0$. By the L\'evy-Khintchine formula \cite{Applebaum,Bertoin}
the characteristic function of $X_t$ is
$\mathbb{E}(e^{i\xi\cdot X_t})=e^{-t\kappa^{2s}|\xi|^{2s}}$, $\xi\in\R^n$, $t\geq0$, 
for some positive constant $\kappa$ that for simplicity we take equal to $1$. For $u\in\mathcal{S}$ set
$T_tu(x)=\mathbb{E}(u(X_t+x))$, $x\in\R^n$, $t\geq0$.
Then, by Fubini's Theorem, $\widehat{T_tu}(\xi)=e^{-t|\xi|^{2s}}\widehat{u}(\xi)$.
Therefore, the function $v(x,t)=T_tu(x)$ solves the fractional diffusion equation
$$\begin{cases}
\partial_tv=-(-\Delta)^sv&\hbox{in}~\R^n\times(0,\infty)\\
v(x,0)=u(x)&\hbox{on}~\R^n.
\end{cases}$$

There is a Markov process corresponding to the fractional powers of the Dirichlet Laplacian $-\Delta_D$
in a smooth bounded domain $\Omega$. The process can be obtained as follows: we first kill a Wiener 
process $W$ at $\tau_\Omega$, the first exit time of $W$ from $\Omega$, and then we subordinate
the killed Wiener process using an $s$-stable subordinator $T_t$. This subordinated process has generator $(-\Delta_D)^s$,
see \cite{Song-Vondracek}.

\vskip 0.2cm
\noindent\texttt{II. Financial mathematics.} For a symmetric $2s$-stable L\'evy process $X_t$
with $X_0=x$ consider the optimal stopping time $\tau$ to maximize the function
$$u(x)=\sup_\tau\mathbb{E}\left[\varphi(X_\tau):\tau<\infty\right]$$
where $\varphi\in C_0(\R^n)$. Then $u$ is a solution to the free boundary problem
\begin{equation}\label{eq:obstacleproblem}
\begin{cases}
    u(x)\geq\varphi(x)&\hbox{in}~\R^n\\
    (-\Delta)^su(x)\geq0&\hbox{in}~\R^n \\
    (-\Delta)^su(x)=0&\hbox{in}~\{u(x)>\varphi(x)\}.
\end{cases}
\end{equation}
This obstacle problem arises as a pricing model for American options
\cite{Cont-Tankov,Silvestre Thesis,Silvestre CPAM}. 

\vskip 0.2cm
\noindent\texttt{III. Elasticity, biology.}
An equivalent formulation of the problem Antonio Signorini posed in \cite{Signorini} consists in finding the configuration of an elastic membrane in equilibrium that stays above some given \textit{thin} obstacle. In mathematical terms,
given $\varphi\in C_0(\R^n)$, the solution to the Signorini problem is the function $U=U(x,y)$, $x\in\R^n$, $y\geq0$, that satisfies
$$\begin{cases}
\partial_{yy}U+\Delta_xU=0&\hbox{in}~\R^n\times(0,\infty) \\
U(x,0)\geq\varphi(x)&\hbox{on}~\R^n \\
\partial_yU(x,0)\leq0&\hbox{on}~\R^n \\
\partial_yU(x,0)=0&\hbox{in}~\{U(x,0)>\varphi(x)\},
\end{cases}$$
see for example \cite{Caffarelli-Salsa-Silvestre,Duvaut-Lions}.
A simple observation gives an equivalent description of the problem as an obstacle problem for the fractional Laplacian.
The solution to $\partial_{yy}U+\Delta_xU=0$ with boundary data $u(x):=U(x,0)$
is given by convolution with the Poisson kernel in the upper half space:
\begin{equation}\label{eq:Poissonsemigroup}
U(x,y)=e^{-y(-\Delta_x)^{1/2}}u(x).
\end{equation}
Taking the derivative of $U$ with respect to $y$ and evaluating it at $y=0$ gives
$$\partial_yU(x,y)\big|_{y=0}=-(-\Delta_x)^{1/2}u(x).$$
Hence, the Signorini problem can be rewritten as
$$\begin{cases}
    \nonumber\partial_{yy}U+\Delta_xU=0&\hbox{in}~\R^n\times(0,\infty) \\
    \label{5}U(x,0)\geq\varphi(x)&\hbox{on}~\R^n \\
    \label{6}(-\Delta_x)^{1/2}U(x,0)\geq0&\hbox{on}~\R^n \\
    \label{7}(-\Delta_x)^{1/2}U(x,0)=0&\hbox{in}~\{U(x,0)>\varphi(x)\}.
\end{cases}$$
In other words, the Signorini problem is equivalent
to the obstacle problem \eqref{eq:obstacleproblem} for $s=1/2$ through the relation $u(x)=U(x,0)$ given by
\eqref{eq:Poissonsemigroup}.

Consider next a Signorini problem where the Laplacian $-\Delta_x$
is replaced by another partial differential operator $L$ in a domain $\Omega\subseteq\R^n$.
For example, $L$ can be the Dirichlet Laplacian $-\Delta_D$ (meaning the elastic membrane is 
kept at zero level on $\partial\Omega$) or the heat operator $\partial_t-\Delta$ (this becomes a model for
semipermeable walls, like a cell membrane on $y=0$, see \cite{Duvaut-Lions}).
The associated Poisson semigroup
$$U(x,y)=e^{-yL^{1/2}}u(x)$$
is the solution to
$$\begin{cases}
\partial_{yy}U-LU=0&\hbox{in}~\Omega\times(0,\infty)\\
U\big|_{y=0}=u&\hbox{on}~\Omega
\end{cases}$$
and satisfies
$$\partial_yU\big|_{y=0}=-L^{1/2}u.$$
Then the Signorini problem for $L$ in place of $-\Delta_x$ can be formulated
for $u=U\big|_{y=0}$ in an equivalent way as an obstacle problem for $L^{1/2}$:
$$\begin{cases}
    u\geq\varphi&\hbox{in}~\Omega \\
    L^{1/2}u\geq0&\hbox{in}~\Omega \\
    L^{1/2}u=0&\hbox{in}~\{u>\varphi\},
\end{cases}$$
see \cite{ACM, Caffarelli-Stinga, Stinga, Stinga-Torrea-SIAM}.

\vskip 0.2cm

Our list of problems above does not pretend to be exhaustive at all. Just to mention some more, there are applications in
fluid mechanics \cite{Caffarelli-Vasseur,Constantin},
fractional kinetics and anomalous diffusion \cite{Metzler-Klafter,Sokolov-Klafter-Blumen,Zaslavsky},
strange kinetics \cite{Shlesinger-Zaslavsky-Klafter},
fractional quantum mechanics \cite{Laskin 1,Laskin 2},
L\'evy processes in quantum mechanics \cite{Petroni-Pusterla},
plasmas \cite{Allen},
electrical propagation in cardiac tissue \cite{Bueno-Orovio},
and biological invasions \cite{Bere}.

\subsection{The method of semigroups}

Consider the situation where we have derived a model (usually a nonlinear PDE problem) that involves a fractional power of some 
partial differential operator $L$. As we saw before, $L$ can be a Laplacian or a heat operator, or even an operator
on a manifold \cite{Banica, DeNapoli-Stinga} or a lattice in the case of discrete models \cite{CRSTV}.
Then we are faced at least with the following basic questions.

\vskip 0.2cm
\noindent\textbf{(I)} \textbf{Definition and pointwise formula for fractional operators.} For a general operator $L$, classical functional analysis gives several ways to define $L^s$ according to its analytical properties. Nevertheless, a pure abstract formula is not useful to treat concrete PDE problems and a more or less explicit pointwise expression for $L^su(x)$ is needed in many cases. The starting point for the \textbf{method of semigroups}
is the formula
$$L^su=\frac{1}{\Gamma(-s)}\int_0^\infty\big(e^{-tL}u-u\big)\,\frac{dt}{t^{1+s}}\quad0<s<1$$
where $\Gamma(-s)=\frac{\Gamma(1-s)}{-s}$ is the Gamma function evaluated at $-s$. Here
$v=e^{-tL}u$ is the \textbf{heat diffusion semigroup} generated by $L$ acting on $u$, namely, $v$ the solution to the heat equation for $L$
with initial temperature $u$:
$$\begin{cases}
\partial_tv=-Lv&\hbox{for}~t>0\\
v\big|_{t=0}=u.
\end{cases}$$
The semigroup formula for $L^s$ is classical, see \cite{Balakrishnan, Kato, Komatsu, Yosida}.
The definition is motivated by the numerical identity
\begin{equation}\label{eq:numericalpositive}
\lambda^s=\frac{1}{\Gamma(-s)}\int_0^\infty(e^{-t\lambda}-1)\,\frac{dt}{t^{1+s}}\quad\hbox{for any}~\lambda\geq0
\end{equation}
that can be easily checked with a simple change of variables.
In a similar way, starting from the numerical identity
\begin{equation}\label{eq:numericalnegative}
\lambda^{-s}=\frac{1}{\Gamma(s)}\int_0^\infty e^{-t\lambda}\,\frac{dt}{t^{1-s}}\quad\hbox{for any}~\lambda>0,~s>0
\end{equation}
we can write down the solution to $L^su=f$ as
$$u=L^{-s}f=\frac{1}{\Gamma(s)}\int_0^\infty e^{-tL}f\,\frac{dt}{t^{1-s}}.$$
Again, this semigroup formula for $L^{-s}$ is classical, see \cite{Balakrishnan, Kato, Komatsu, Yosida}.
It turns out these are quite concrete and useful ways of defining and understanding fractional operators.
Indeed, when a heat kernel is available for the semigroup $e^{-tL}$, then pointwise formulas
for both positive and negative powers of $L$ can be obtained, see
\cite{Betancor, BDS, Caffarelli-Stinga, FiveGuys, CRS, CRSTV, Marta-Torrea, DeNapoli-Stinga,
Feo-Stinga-Volzone, Maldonado-Stinga, Roncal-Stinga, Stinga, Stinga-Torrea CPDE,
Stinga-Torrea-SIAM, Stinga-Torrea JFA, Stinga-Volzone}. For degenerate cases like the usual derivative
or discrete derivatives see \cite{Abadias,Bernardis}.

In this chapter we will explain how these formulas work
only for the case of $L=-\Delta$, as developed in \cite{Stinga,Stinga-Torrea CPDE}.
Sections \ref{sec:positive} and \ref{sec:negative}
are devoted to show how the semigroup definitions of $(-\Delta)^s$ and $(-\Delta)^{-s}$ follow from
the above-mentioned numerical formulas and how, with the help of the classical heat semigroup kernel, one can
obtain the well-known nonlocal pointwise formula
$$(-\Delta)^su(x)=c_{n,s}\PV\int_{\R^n}\frac{u(x)-u(z)}{|x-z|^{n+2s}}\,dz$$
and similarly for $(-\Delta)^{-s}f(x)$.
Obviously, these formulas are very well-known \cite{Landkof,Stein} and can be deduced through several other techniques.
Nevertheless, we present the semigroup ideas in this simple case so the reader can use them in other applications.
\vskip 0.2cm
\noindent\textbf{(II)} \textbf{The nonlocal nature.} 
The fractional Laplacian is a nonlocal operator. Indeed, the value of $(-\Delta)^su(x)$ for a given $x\in\R^n$ depends on the values of $u$ at infinity. 
Also, in general, if $u$ has compact support then $(-\Delta)^su$ has noncompact support.
This basic property may create some issues. For example, the classical local PDE methods from the calculus of variations
based on integration by parts and localization using test functions cannot be directly applied to the study of nonlinear problems for $(-\Delta)^s$. Even the notion of viscosity solution needs to be redefined to take into account the values of solutions at infinity \cite{Caffarelli-Silvestre-CPAM}. L.~A.~Caffarelli and L.~Silvestre showed in \cite{Caffarelli-Silvestre CPDE} that any fractional power of the Laplacian can be characterized as an operator that maps a Dirichlet boundary condition to a Neumann-type condition via an extension PDE problem. From a probabilistic point of view, the extension problem corresponds to the property that all symmetric stable processes can be obtained as traces of degenerate Bessel diffusion processes, see \cite{Molchanov-Ostrovskii}.
Consider the function $U=U(x,y):\R^n\times[0,\infty)\to\R$ that solves the degenerate elliptic boundary value problem
$$\begin{cases}
\Delta_x U+\frac{a}{y}~U_y+U_{yy}=0&x\in\R^n,~y>0\\
U(x,0)=u(x)&x\in\R^n
\end{cases}$$
where $a=1-2s$. Then, for any $x\in\R^n$,
$$-\lim_{y\to0^+}y^{1-2s}U_y(x,y)=c_s(-\Delta)^su(x)$$
see \cite{Caffarelli-Silvestre CPDE}. The constant $c_s>0$ was computed explicitly for the first time in \cite{Stinga,Stinga-Torrea CPDE}.
We can interpret this result as saying that the new variable $y$ added to extend $u$ to the upper half space through $U$ encodes the values of $u$ at infinity needed to compute $(-\Delta)^su$. The extension problem localizes the fractional Laplacian: it is enough to know $U$ in some upper half ball around $(x,0)$ to already get $(-\Delta)^su(x)$. The nonlinear problems for the nonlocal fractional Laplacian can then be localized by adding a new variable. Now one can exploit the classical PDE tools and ideas that are available
for these equations \cite{Fabes}. The work of Caffarelli and Silvestre \cite{Caffarelli-Silvestre CPDE} presented applications
to Harnack inequalities and monotonicity formulas for $(-\Delta)^s$ by applying such local PDE techniques
in the extension problem.
Since then, \cite{Caffarelli-Silvestre CPDE} has created an explosion of results
on problems with fractional Laplacians, see \cite{Caffarelli-Salsa-Silvestre,Caffarelli-Vasseur} for a couple of important examples.

In general, fractional power operators $L^s$ are nonlocal operators. It would be very useful in applications to have an analogous to the
Caffarelli--Silvestre characterization for $L^s$ as a Dirichlet-to-Neumann map via an extension problem. This open problem
was solved in \cite{Stinga, Stinga-Torrea CPDE}. The author and Torrea discovered 
an extension problem for fractional operators on Hilbert spaces. Later on, J.~E.~Gal\'e, P.~J.~Miana and the author
found an extension problem characterization for fractional powers of operators in Banach spaces and,
more generally, generators of integrated semigroups, see \cite{Gale-Miana-Stinga}. In addition, \cite{Gale-Miana-Stinga}
included the case of complex fractional power operators.
The semigroup point of view turned out to be fundamental. As a matter of fact, when $L=-\Delta$ in
\cite{Gale-Miana-Stinga, Stinga, Stinga-Torrea CPDE} then one recovers the extension PDE of \cite{Caffarelli-Silvestre CPDE}.
Some of the main novelties of \cite{Gale-Miana-Stinga, Stinga, Stinga-Torrea CPDE} were the analysis
of the extension PDE by means of Bessel functions and the explicit semigroup formulas for the solution
\begin{align*}
U(y) &= \frac{y^{2s}}{4^s\Gamma(s)}\int_0^\infty e^{-y^2/(4t)}e^{-tL}u\,\frac{dt}{t^{1+s}} \\
&= \frac{1}{\Gamma(s)}\int_0^\infty e^{-y^2/(4t)}e^{-tL}(L^su)\,\frac{dt}{t^{1-s}}.
\end{align*}
These were new even for the case of the fractional Laplacian. As it could be expected, these general extension problems
found many applications such as free boundary problems \cite{Allen, ACM},
fractional derivatives \cite{Bernardis}, master equations \cite{BDS,Marta-Torrea,Stinga-Torrea-SIAM},
fractional elliptic PDEs \cite{Caffarelli-Stinga,Stinga-Zhang,Yu},
fractional Laplacians on manifolds \cite{Banica, Caffarelli-Sire, Chamorro, DeNapoli-Stinga, Ferrari-Franchi}
and in infinite dimensions \cite{Novaga},
symmetrization \cite{DiBlasio-Volzone,Feo-Stinga-Volzone}, nonlocal Monge--Amp\`ere equations \cite{Maldonado-Stinga},
numerical analysis \cite{Nochetto}, biology \cite{Stinga-Volzone} and inverse problems \cite{Garcia,Ghosh}.

On the other hand, an extension problem for higher powers of fractional operators in Hilbert spaces
using heat semigroups was proved in \cite{Roncal-Stinga}, see also \cite{Yang} for
the particular case of the fractional Laplacian on $\R^n$.
The fractional powers of the Laplacian can also be characterized by means of a wave extension problem,
see \cite{Kemppainen-Sjogren-Torrea}. In such scenario 
the wave and Schr\"odinger groups (instead of the heat semigroup) play a key, fundamental role.

We will not go into more details about all these general cases here,
but we will only show how the semigroup ideas, techniques and formulas of \cite{Gale-Miana-Stinga,Stinga,Stinga-Torrea CPDE}
work for the extension problem in the fractional Laplacian case, see Section \ref{sec:extension}.
Applications to Harnack inequalities and derivative estimates for $s$-harmonic functions
are given in Section \ref{sec:Harnack} by following \cite{Caffarelli-Salsa-Silvestre,Caffarelli-Silvestre CPDE}.

\vskip 0.2cm
\noindent\textbf{(III)} \textbf{Regularity theory for fractional operators.} 
Clearly the Fourier transform definition of the fractional Laplacian does not seem to be the most useful formulation
to prove regularity estimates in H\"older and Zygmund spaces. One strategy that has been followed
for this problem is to make heavy use of the pointwise formulas for $(-\Delta)^s$ and $(-\Delta)^{-s}$,
see, for example, \cite{Silvestre Thesis, Silvestre CPAM}. This is only natural as pointwise formulas clearly allow us to
handle differences of the form $|(-\Delta)^su(x_1)-(-\Delta)^su(x_2)|$.

We present here a semigroup method towards proving regularity estimates, where only
the semigroup formulas for the fractional operators are needed. We will first show that H\"older
and Zygmund spaces are characterized by means of the growth of time derivatives of
the heat semigroup $\partial_t^ke^{t\Delta}$,
see Section \ref{sec:spaces}. The proof of such characterization is obviously nontrivial.
It will be shown in Section \ref{sec:regularity} how the semigroup descriptions of
H\"older--Zygmund spaces and fractional Laplacians allow for a quick, elegant proof of H\"older and Schauder estimates.
We believe this is the first time these results have been presented and proved in such a systematic, complete way
for the case of the fractional Laplacian.

If we now think about fractional powers of other differential operators $L$,
we may ask for the ``right'' Schauder estimates for $L^s$.
More precisely, what is the proper/adapted H\"older space to look for regularity properties of $L^s$?
The semigroup approach then comes at hand: one can define
regularity spaces associated to $L$ in terms of the growth of heat semigroups $\partial_t^ke^{-tL}$
in complete analogy to the case of the classical H\"older--Zygmund spaces.
As we mentioned, passing from a semigroup formulation to a pointwise description
of such spaces is a nontrivial task that must be carefully handled in each particular situation.
Despite this, the great advantage is that
the regularity properties of fractional powers $L^s$ on these spaces will follow at once using the semigroup representations.
See, for example, \cite{Caffarelli-Stinga} for the fractional Laplacian,
\cite{Ma-Stinga-Torrea-Zhang, Stinga-Torrea JFA} for Schr\"odinger operators $L=-\Delta+V$,
\cite{Gatto-Urbina,Liu-Sjogren} for the Ornstein--Uhlenbeck operator $L=-\Delta+\nabla\cdot x$,
\cite{BDS,Marta-Torrea,Stinga-Torrea-SIAM} for fractional powers of parabolic operators,
\cite{Roncal-Stinga} for the fractional Laplacian on the torus, and
\cite{Betancor} for Bessel operators and radial solutions to the fractional Laplacian.

\vskip 0.2cm
As we said before, the fractional Laplacian is a classical object in mathematics, and many of the results
we will present here can be proved in several different ways and with other techniques.
An exhaustive list of classical and modern references dealing with them is out of the scope
of this chapter and the reader is invited to explore the references mentioned at the beginning of this section
as well as those contained in other chapters of this volume.
 
\section{Fractional Laplacian: semigroups, pointwise formulas
and limits}\label{sec:positive}

Recall the Fourier transform definition of the fractional Laplacian given in 
\eqref{eqFractionalLaplacian:Fouriertransformdefinition}.
It is obvious that
$(-\Delta)^0u=u$, $(-\Delta)^1u=-\Delta u$ and, for any $s_1,s_2$ we have
$(-\Delta)^{s_1}\circ(-\Delta)^{s_2}u=(-\Delta)^{s_1+s_2}u$.
Even though $|\xi|^{2s}\widehat{u}(\xi)$ is a well defined function of $\xi\in\R^n$, we still have
$$(-\Delta)^su\notin\S$$
because $|\xi|^{2s}$ creates a singularity at $\xi=0$.
On the other hand, \eqref{eqFractionalLaplacian:Fouriertransformdefinition} implies that 
for any multi-index $\gamma\in\N_0^n$,
\begin{equation}\label{eqFractionalLaplacian:commute}
D_x^\gamma(-\Delta)^s=(-\Delta)^sD_x^\gamma.
\end{equation}
In particular, if $u\in\S$ then $(-\Delta)^su\in C^\infty(\R^n)$.

To compute $(-\Delta)^su(x)$ for each point $x\in\R^n$
one could try to take the inverse Fourier transform in \eqref{eqFractionalLaplacian:Fouriertransformdefinition}.
In fact, since $|\xi|^{2s}\widehat{u}(\xi)\in L^1(\R^n)$, one can make sense to
$(-\Delta)^su(x)=\mathcal{F}^{-1}(|\xi|^{2s}\widehat{u}(\xi))(x)$.
But here we are going to avoid this and, instead, apply the method of semigroups.
If we choose $\lambda=|\xi|^2$, for $\xi\in\R^n$, in the numerical formula \eqref{eq:numericalpositive},
multiply it by $\widehat{u}(\xi)$ and recall \eqref{eqFractionalLaplacian:Fouriertransformdefinition}, then
$$\widehat{(-\Delta)^su}(\xi)=|\xi|^{2s}\widehat{u}(\xi)=\frac{1}{\Gamma(-s)}\int_0^\infty(e^{-t|\xi|^2}\widehat{u}(\xi)-\widehat{u}(\xi))\,\frac{dt}{t^{1+s}}.$$
Thus, by inverting the Fourier transform,
we obtain the \textbf{semigroup formula for the fractional Laplacian}
(see \cite{Balakrishnan,Kato,Komatsu,Yosida}, also \cite{Gale-Miana-Stinga,Stinga, Stinga-Torrea CPDE})
\begin{equation}\label{eqFractionalLaplacian:semigroupformula}
(-\Delta)^su(x)=\frac{1}{\Gamma(-s)}\int_0^\infty\big(e^{t\Delta}u(x)-u(x)\big)\,\frac{dt}{t^{1+s}}.
\end{equation}

The family of operators $\{e^{t\Delta}\}_{t\geq0}$ is the classical heat diffusion semigroup generated by $\Delta$.
Consider the solution $v=v(x,t)$, for $x\in\R^n$ and $t\geq0$, of the heat equation on the whole space $\R^n$
with initial temperature $u$:
$$\begin{cases}
\partial_tv=\Delta v&\hbox{for}~x\in\R^n,~t>0\\
v(x,0)=u(x)&\hbox{for}~x\in\R^n.
\end{cases}$$
If we apply the Fourier transform in the variable $x$ for each fixed $t$ then
\begin{equation}\label{eq:heat semigroup pre Fourier}
\widehat{v}(\xi,t)=e^{-t|\xi|^2}\widehat{u}(\xi)=\widehat{e^{t\Delta}u}(\xi)
\end{equation}
so that $u\longmapsto e^{t\Delta}u$ is the solution operator.
It is well known that
$$v(x,t)\equiv e^{t\Delta}u(x)=G_t\ast u(x)=\int_{\R^n}G_t(x-z)u(z)\,dz$$
where $G_t(x)$ is the \textbf{Gauss--Weierstrass heat kernel}:
\begin{equation}\label{eqFractionalLaplacian:heatkernel}
G_t(x)=\frac{1}{(4\pi t)^{n/2}}e^{-|x|^2/(4t)}.
\end{equation}
Observe that $G_t$ defines an approximation
of the identity. Moreover,
\begin{equation}\label{eqFractionalLaplacian:integral heat kernel}
e^{t\Delta}1(x)=\int_{\R^n}G_t(x)\,dx\equiv1\quad\hbox{for any}~x\in\R^n,~t>0.
\end{equation}

\begin{rem}[Maximum principle]
The semigroup formula \eqref{eqFractionalLaplacian:semigroupformula} and the positivity of the 
heat kernel \eqref{eqFractionalLaplacian:heatkernel} easily imply the maximum principle for the fractional Laplacian.
Indeed, if $u\geq0$ and $u(x_0)=0$ at some point $x_0\in\R^n$ then
$$(-\Delta)^su(x_0)=\frac{1}{\Gamma(-s)}\int_0^\infty e^{t\Delta}u(x_0)\,\frac{dt}{t^{1+s}}\leq0.$$
Moreover, $(-\Delta)^su(x_0)=0$ if and only if $e^{t\Delta}u(x_0)=0$, that is, only when $u\equiv0$.
For another proof using pointwise formulas, see \cite{Silvestre Thesis, Silvestre CPAM}. For maximum
principles for fractional powers of elliptic operators using semigroups, see \cite{Stinga-Zhang}.
\end{rem}

The semigroup formula \eqref{eqFractionalLaplacian:semigroupformula} and the heat
kernel \eqref{eqFractionalLaplacian:heatkernel} permit us to compute the pointwise formula for the fractional Laplacian.
The technique avoids the inverse Fourier transform
in \eqref{eqFractionalLaplacian:Fouriertransformdefinition} and gives the constants explicitly.

\begin{thm}[Pointwise formulas]\label{thmFractionalLaplacian:pointwiseformulaS}
Let $u\in\S$, $x\in\R^n$ and $0<s<1$.
\begin{enumerate}[$(i)$]
\item If $0<s<1/2$ then
$$(-\Delta)^s u(x)=c_{n,s}\int_{\R^n}\frac{u(x)-u(z)}{|x-z|^{n+2s}}\,dz$$
and the integral is absolutely convergent.
\item If $1/2\leq s<1$ then, for any $\delta>0$,
\begin{align*}
(-\Delta)^s u(x) &= c_{n,s}\lim_{\vep\to0^+}\int_{|x-z|>\vep}\frac{u(x)-u(z)}{|x-z|^{n+2s}}\,dz \\
&= c_{n,s}\int_{\R^n}\frac{u(x)-u(z)-\nabla u(x)\cdot(x-z)\chi_{|x-z|<\delta}(z)}{|x-z|^{n+2s}}\,dz
\end{align*}
where the second integral is absolutely convergent.
\end{enumerate}
The constant $c_{n,s}>0$ in the formulas above is explicitly given by
\begin{equation}\label{eq:cns}
c_{n,s}=\frac{4^s\Gamma(n/2+s)}{|\Gamma(-s)|\pi^{n/2}}=\frac{s(1-s)4^s\Gamma(n/2+s)}{|\Gamma(2-s)|\pi^{n/2}}.
\end{equation}
In particular, $c_{n,s}\sim s(1-s)$ as $s\to0^+$ and $s\to1^-$.
\end{thm}

\begin{proof}[Sketch of proof]
The detailed proof using the heat kernel can be found in \cite{Stinga, Stinga-Torrea CPDE}.
We write down the formula for $e^{t\Delta}u(x)$ in \eqref{eqFractionalLaplacian:semigroupformula}
and use \eqref{eqFractionalLaplacian:integral heat kernel} to get
\begin{equation}\label{eqFractionalLaplacian:preFubini}
(-\Delta)^su(x)= \frac{1}{|\Gamma(-s)|}\int_0^\infty\int_{\R^n}G_t(x-z)(u(x)-u(z))\,dz\,\frac{dt}{t^{1+s}}.
\end{equation}
By \eqref{eqFractionalLaplacian:heatkernel} and the change of variables $r=|x-z|^2/(4t)$,
\begin{equation}\label{eqFractionalLaplacian:kernelcomputation}
\frac{1}{|\Gamma(-s)|}\int_0^\infty G_t(x-z)\,\frac{dt}{t^{1+s}} = c_{n,s}\cdot\frac{1}{|x-z|^{n+2s}}.
\end{equation}
When $0<s<1/2$ the double integral in \eqref{eqFractionalLaplacian:preFubini}
is absolutely convergent and Fubini's theorem gives $(i)$.
If $1/2\leq s<1$ then one needs to use the fact that, for any $i=1,\ldots,n$ and $0<\varepsilon<\delta$,
$$\int_{|z|<\delta}z_iG_t(z)\,dz=\int_{\varepsilon<|z|<\delta}z_iG_t(z)\,dz=0$$
for all $t>0$, for the gradient term to appear in $(ii)$.
\end{proof}

The pointwise formulas in Theorem \ref{thmFractionalLaplacian:pointwiseformulaS}
are valid for $u\in\S$. However, the integrals are well defined for less regular functions.
As a matter of fact, we can relax the requirement on $u$ at infinity by asking that
\begin{equation}\label{eqFractionalLaplacian:integrability}
\|u\|_{L_s}:=\int_{\R^n}\frac{|u(x)|}{1+|x|^{n+2s}}\,dx<\infty.
\end{equation}
When $0<s<1/2$, we only need $u$ to be $C^{2s+\varepsilon}$ at $x$, for $2s+\varepsilon\leq1$, for the
singular part of the integral (that is, when $z$ is near $x$) to be finite:
$$\int_{|x-z|<1}\frac{|u(x)-u(z)|}{|x-z|^{n+2s}}\,dz\leq [u]_{C^{2s+\vep}(x)}\int_{|x-z|<1}\frac{|x-z|^{2s+\vep}}{|x-z|^{n+2s}}\,dz<\infty.$$
Similarly, when $1/2\leq s<1$, the requirement $u\in C^{1,2s+\varepsilon-1}$ at $x$, for $2s+\varepsilon-1\leq1$ would suffice as well. 

To define the fractional Laplacian for less regular functions,
we need to understand what is 
$(-\Delta)^s$ in the sense of distributions. The fractional Laplacian is a symmetric operator on $L^2(\R^n)$: when $u,f\in\S$,
$$\int_{\R^n}(-\Delta)^su(x)f(x)\,dx=\int_{\R^n}|\xi|^{2s}\widehat{u}(\xi)\widehat{f}(\xi)\,d\xi=\int_{\R^n}u(x)(-\Delta)^sf(x)\,dx.$$
One may then think in the following way. For $u\in\S'$
(a tempered distribution) and $f\in\S$ one could define the distribution $(-\Delta)^su$ as
$\langle(-\Delta)^su,f\rangle=\langle u,(-\Delta)^sf\rangle$.
The problem here is that $(-\Delta)^sf\notin\S$, so this identity makes no sense for $u\in\S'$.
First we need to characterize the set $(-\Delta)^s(\S)$, see \cite{CRSTV, Silvestre Thesis, Silvestre CPAM}.

\begin{lem}\label{lemFractionalLaplacian:classSs}
Let $f\in\S$. Then $(-\Delta)^sf$ belongs to the class $\S_s$ defined by
$$\S_s=\big\{\psi\in C^\infty(\R^n):(1+|x|^{n+2s})D^\gamma\psi(x)\in L^\infty(\R^n),~\hbox{for every}~\gamma\in\N_0^n\big\}.$$
\end{lem}

The class $\S_s$ of Lemma \ref{lemFractionalLaplacian:classSs} 
is endowed with the topology induced by the countable family of seminorms
$$\rho_\gamma(\psi)=\sup_{x\in\R^n}|(1+|x|^{n+2s})D^\gamma\psi(x)|\quad\gamma\in\N_0^n.$$
Denote by $\S_s'$ the dual space of $\S_s$. Observe that
$\S\subset\S_s$, so that $\S_s'\subset\S'$.
The suitable space for the distributional definition of the fractional Laplacian is $\S_s'$.

\begin{defn}
Let $u\in\S_s'$. We define $(-\Delta)^su\in\S'$ as
$$\big((-\Delta)^su\big)(f)=u((-\Delta)^sf)\quad\hbox{for every}~f\in\mathcal{S}.$$
In terms of pairings, we write $\langle (-\Delta)^su,f\rangle_{\S',\S}=\langle u,(-\Delta)^sf\rangle_{\S_s',\S_s}$.
\end{defn}

If $u\in\S$ then this distributional definition coincides with the one given in terms of the Fourier transform.
Also, $(-\Delta)^s$ maps $\S_s'$ into $\S'$ continuously.
Recall the definition of the space $L_s$ given in \eqref{eqFractionalLaplacian:integrability}.
We have $L_s=L^1_{\mathrm{loc}}(\R^n)\cap\S_s'$. The proof of the following result is based on an approximation
argument and the details can be found in \cite{Silvestre Thesis, Silvestre CPAM}.

\begin{thm}[Pointwise formula for less regular functions]\label{thmFractionaLapalcian:pointwiselessregular}
Let $\Omega$ be an open subset of $\R^n$ and let $u\in L_s$, $0<s<1$. If $u\in C^{2s+\vep}(\Omega)$ (or $C^{1,2s+\vep-1}(\Omega)$
if $s\geq1/2$) for some $\vep>0$ then $(-\Delta)^su$ is a continuous function in $\Omega$ and
$(-\Delta)^su(x)$ is given by the pointwise formulas of Theorem \ref{thmFractionalLaplacian:pointwiseformulaS},
for every $x\in\Omega$.
\end{thm}

\begin{rem}[Semigroup formula for less regular functions]
It is an exercise to verify that the semigroup formula \eqref{eqFractionalLaplacian:semigroupformula},
that was initially derived for functions $u\in\mathcal{S}$, also holds for the class of less regular functions $u$
considered in Theorem \ref{thmFractionaLapalcian:pointwiselessregular}, for all $x\in\Omega$.
Indeed, it is easy to work out the computations in the proof of Theorem \ref{thmFractionalLaplacian:pointwiseformulaS} in a reverse order,
namely, by starting with the pointwise integro-differential formula for $(-\Delta)^su(x)$ and using the heat kernel
identity \eqref{eqFractionalLaplacian:kernelcomputation} to end up with \eqref{eqFractionalLaplacian:semigroupformula}. 
\end{rem}

We turn our attention to the pointwise limits for $s\to1^-$ and $s\to0^+$ in the case of less regular functions.
The explicit value of the constant $c_{n,s}$ in \eqref{eq:cns}, that we found through the method of semigroups, plays
a crucial role. Observe as well that we require minimal regularity on $u$ for $(-\Delta)^su(x)$
to be defined through the integral formula and for the \textit{pointwise} limits to have sense.

\begin{thm}
If $u\in C^2(B_2(x))\cap L^\infty(\R^n)$ at some $x\in\R^n$ then
$$\lim_{s\to1^-}(-\Delta)^su(x)=-\Delta u(x).$$
\end{thm}

\begin{proof}[Sketch of proof]
The full details of the proof can be found in \cite{Stinga, Stinga-Torrea CPDE}.
We can write $(-\Delta)^su(x)$ as
$$c_{n,s}\int_{|x-z|>\delta}\frac{u(x)-u(z)}{|x-z|^{n+2s}}\,dz
+c_{n,s}\int_{|x-z|<\delta}\frac{u(x)-u(z)-\nabla u(x)\cdot(x-z)}{|x-z|^{n+2s}}\,dz.$$
Since $u$ is bounded the first term above converges to zero as $s\to1^-$.
The second term can be written as
\begin{align*}
I-II&:=c_{n,s}\int_0^\delta r^{-1-2s}\int_{|z'|=1}\bigg[\frac{r^2}{2}\langle D^2u(x)z',z'\rangle-R_1u(x,rz')\bigg]\,dS_{z'}\,dr \\
&\qquad -c_{n,s}\int_0^\delta r^{-1-2s}\int_{|z'|=1}\frac{r^2}{2}\langle D^2u(x)z',z'\rangle\,dS_{z'}\,dr
\end{align*}
where $R_1u(x,rz')=u(x-rz')-u(x)+\nabla u(x)\cdot(rz')$.
By the regularity of $u$, $|I|\leq C\vep$ and using that
$$\int_{|z'|=1}\langle D^2u(x)z',z'\rangle\,dz'=\frac{(n/2+1)\pi^{n/2}}{\Gamma(n/2+2)}\Delta u(x)$$
the conclusion follows.
\end{proof}

Recall the definition of the space $L_s$ from \eqref{eqFractionalLaplacian:integrability}, for $0\leq s\leq1$.

\begin{thm}
If $u\in  C^\alpha(B_2(x))\cap L_0$ for some $x\in\R^n$ and $0<\alpha<1$ then
$$\lim_{s\to0^+}(-\Delta)^su(x)=u(x).$$
\end{thm}

\begin{proof}[Sketch of proof]
The complete details of the proof are found in \cite{Stinga}.
We have
$$(-\Delta)^su(x)=c_{n,s}\int_{|x-z|<R}\frac{u(x)-u(z)}{|x-z|^{n+2s}}\,dz
+c_{n,s}\int_{|x-z|>R}\frac{u(x)-u(z)}{|x-z|^{n+2s}}\,dz$$
for $R=1+|x|$.
Since $u$ is regular, the first term above converges to zero as $s\to0^+$.
The second term is split into two: one with $u(x)$ (that will converge to $u(x)$
as $s\to0^+$) and the other one with $u(y)$ (that will converge to zero as $s\to0^+$).
\end{proof}

\section{Inverse fractional Laplacian: semigroups, pointwise formula and the Poisson problem}\label{sec:negative}

If we apply the Fourier transform to solve the Poisson equation
$$(-\Delta)^su=f\quad\hbox{in}~\R^n$$
we find that $|\xi|^{2s}\widehat{u}(\xi)=\widehat{f}(\xi)$.
The inverse of the fractional Laplacian, or \textbf{negative power of the Laplacian} $(-\Delta)^{-s}$, $s>0$, is defined 
for $f\in\S$ as
\begin{equation}\label{eqFractionalLaplacian:negativepowersFourier}
\widehat{(-\Delta)^{-s}f}(\xi)=|\xi|^{-2s}\widehat{f}(\xi)\quad\hbox{for}~\xi\neq0.
\end{equation}
In principle, we need the restriction $0<s<n/2$ because when $s\geq n/2$ the multiplier $|\xi|^{-2s}$
does not define a tempered distribution, see \cite{Silvestre Thesis, Silvestre CPAM, Stein}.
This operator is also known as the \textit{fractional integral operator} in the harmonic analysis literature \cite{Duo,Stein}.

To find a pointwise expression for $(-\Delta)^{-s}f(x)$ at a point $x\in\R^n$ one could try to compute the inverse
Fourier transform in \eqref{eqFractionalLaplacian:negativepowersFourier}. This is a delicate task as
the Fourier multiplier $|\xi|^{-2s}$ is not in $L^2(\R^n)$, see \cite{Stein}.
Instead, we apply the method of semigroups.
We start by choosing $\lambda=|\xi|^2$, $\xi\neq0$, in the numerical formula \eqref{eq:numericalnegative}
to find that, for a.e. $\xi\in\R^n$,
$$\widehat{(-\Delta)^{-s}f}(\xi)=\frac{1}{\Gamma(s)}\int_0^\infty e^{-t|\xi|^2}\widehat{f}(\xi)\,\frac{dt}{t^{1-s}}.$$
Therefore, by inverting the Fourier transform,
we obtain the \textbf{semigroup formula for the inverse fractional Laplacian}
(see \cite{Balakrishnan,Kato,Komatsu,Stein,Yosida}, also \cite{Stinga})
\begin{equation}\label{eqFractionalLaplacian:fractionalintegral}
(-\Delta)^{-s}f(x)=\frac{1}{\Gamma(s)}\int_0^\infty e^{t\Delta}f(x)\,\frac{dt}{t^{1-s}}.
\end{equation}

The proof of the following result using heat kernels when $0<s<n/2$ is contained in \cite{Stinga}.
The proof of the case $s=n/2$ is an original contribution of the author for this volume.
Indeed, it has not been published elsewhere yet.

\begin{thm}[Fundamental solution]\label{thmFractionalLaplacian:fundamentalsolution}
Let $f\in\S$, $x\in\R^n$ and $0<s\leq n/2$. In the case when $s=n/2$ assume in addition that
$\displaystyle\int_{\R^n}f=0$. Then
\begin{equation}\label{eqFractionalLaplacian:negativeconvolution}
(-\Delta)^{-s}f(x)=\int_{\R^n}K_{-s}(x-z)f(z)\,dz.
\end{equation}
Here
$$K_{-s}(x)=\begin{cases}
\displaystyle c_{n,-s}\frac{1}{|x-z|^{n-2s}}&\hbox{if}~0<s<n/2 \\
\smallskip
\displaystyle\frac{1}{\Gamma(n/2)(4\pi)^{n/2}}(-2\log|x|-\gamma)&\hbox{if}~s=n/2
\end{cases}$$
where
$$\gamma=-\int_0^\infty e^{-r}\log r\,dr\approx 0.577215$$
is the Euler--Mascheroni constant and
$$c_{n,-s}=\frac{\Gamma(n/2-s)}{4^s\Gamma(s)\pi^{n/2}}\quad\hbox{for}~0<s<n/2.$$
\end{thm}

The reader should compare the explicit constant $c_{n,-s}$ in Theorem \ref{thmFractionalLaplacian:fundamentalsolution}
(that we found through the method of semigroups, see \cite{Stinga}) with the constant $c_{n,s}$ 
for the fractional Laplacian given in \eqref{eq:cns}.

\begin{proof}[Proof of Theorem \ref{thmFractionalLaplacian:fundamentalsolution} in the case $2s=n$]
As we mentioned above, this proof is an original contribution of the author for this volume.
From \eqref{eqFractionalLaplacian:fractionalintegral} and \eqref{eqFractionalLaplacian:heatkernel},
the change of variables $r=1/(4t)$ and the fact that $f$ has zero mean,
\begin{equation}\label{eq:preFubini2}
\begin{aligned}
(-\Delta)^{-s}f(x) &= \frac{1}{\Gamma(n/2)(4\pi)^{n/2}}\int_0^\infty\int_{\R^n}e^{-r|z|^2}f(x-z)\,dz\,\frac{dr}{r} \\
 &= \frac{1}{\Gamma(n/2)(4\pi)^{n/2}}\int_0^\infty\int_{\R^n}\big(e^{-r|z|^2}-\chi_{(0,1)}(r)\big)f(x-z)\,dz\,\frac{dr}{r}.
\end{aligned}
\end{equation}
The second double integral in \eqref{eq:preFubini2}
is absolutely convergent. Indeed,
\begin{equation*}
\begin{aligned}
&\int_0^\infty|e^{-r|z|^2}-\chi_{(0,1)}(r)|\,\frac{dr}{r}
= \int_0^{|z|^2}(1-e^{-r})\,\frac{dr}{r}+\int_{|z|^2}^\infty e^{-r}\,\frac{dr}{r} \\
&= \int_0^{|z|^2}(1-e^{-r})\frac{d}{dr}(\log r)\,dr+\int_{|z|^2}^\infty e^{-r}\frac{d}{dr}(\log r)\,dr \\
&= (1-e^{-|z|^2})\log(|z|^2)-\int_0^{|z|^2}e^{-r}\log r\,dr-e^{-|z|^2}\log(|z|^2)+\int_{|z|^2}^\infty e^{-r}\log r\,dr \\
&\leq  C+C|\log|z||\in L^1_{\mathrm{loc}}(\R^n).
\end{aligned}
\end{equation*}
Thus we can apply Fubini's theorem in \eqref{eq:preFubini2}. By following the computation
we just did above we get the formula for the kernel:
\begin{align*}
K_{-n/2}(z) &= \frac{1}{\Gamma(n/2)(4\pi)^{n/2}}\bigg[\int_0^1\big(e^{-r|z|^2}-1\big)\,\frac{dr}{r}
+\int_1^\infty e^{-r|z|^2}\,\frac{dr}{r}\bigg] \\
&= \frac{1}{\Gamma(n/2)(4\pi)^{n/2}}\bigg[-\log(|z|^2)+\int_0^\infty e^{-r}\log r\,dr\bigg].
\end{align*}
\end{proof}

\begin{thm}[Distributional solvability]\label{thmFractionalLaplacian:distributionalsolvability}
Let $f\in L^\infty(\R^n)$ with compact support and $0<s<\min\{1,n/2\}$. Define
$$u(x)=(-\Delta)^{-s}f(x)=\frac{1}{\Gamma(s)}\int_0^\infty e^{t\Delta}f(x)\,\frac{dt}{t^{1-s}}.$$
Then $u$ is given by the pointwise formula \eqref{eqFractionalLaplacian:negativeconvolution},
$u\in L^\infty(\R^n)$ with
$$\|u\|_{L^\infty(\R^n)}\leq C_{n,s}\|f\|_{L^\infty(\R^n)}$$
for some positive constant $C_{n,s}$, and
$$|u(x)|\to 0\quad\hbox{as}~|x|\to\infty.$$
In addition, $(-\Delta)^su=f$ in the sense of distributions.
\end{thm}

The proof of Theorem \ref{thmFractionalLaplacian:distributionalsolvability} follows the exact same lines
as the proof of the one dimensional case presented in \cite[Theorem~9.9]{CRSTV}.

\section{Extension problem: semigroup approach, weak formulation}\label{sec:extension}

As we mentioned in the introduction, the extension problem for the fractional Laplacian 
is a characterization of $(-\Delta)^s$ as the Dirichlet-to-Neumann map for
a local degenerate elliptic PDE. This localization technique was introduced and exploited in the PDE context by
Caffarelli and Silvestre \cite{Caffarelli-Silvestre CPDE}.
The method of semigroups for this problem that the author and Torrea developed in
\cite{Stinga, Stinga-Torrea CPDE} (see also \cite{Gale-Miana-Stinga}) provided new insights,
explicit formulas for the solution, the useful Bessel functions analysis and, ultimately, a unified approach.
To present the extension problem, for $0<s<1$, we let
$$a=1-2s\in(-1,1).$$

\begin{thm}[Extension problem for positive powers]\label{thmFractionalLaplacian:extensionproblem}
Let $u\in\S$. The unique solution $U=U(x,y):\R^n\times[0,\infty)\to\R$ to
\begin{equation}\label{eqFractionalLaplacian:extensionequation}
\begin{cases}
\Delta U+\frac{a}{y}U_y+U_{yy}=0&\hbox{in}~\R^n\times(0,\infty)\\
U(x,0)=u(x)&\hbox{on}~\R^n
\end{cases}
\end{equation}
that weakly vanishes as $y\to\infty$ is given by the following formulas
\begin{equation}\label{eqFractionalLaplacian:U}
\begin{aligned}
U(x,y) &= \frac{y^{2s}}{4^s\Gamma(s)}\int_0^\infty e^{-y^2/(4t)}e^{t\Delta}u(x)\,\frac{dt}{t^{1+s}} \\
&= \frac{1}{\Gamma(s)}\int_0^\infty e^{-t}e^{-\frac{y^2}{4t}\Delta}u(x)\,\frac{dt}{t^{1-s}} \\
&= \frac{1}{\Gamma(s)}\int_0^\infty e^{-y^2/(4t)}e^{t\Delta}((-\Delta)^su)(x)\,\frac{dt}{t^{1-s}} \\
&= \frac{\Gamma(n/2+s)}{\Gamma(s)\pi^{n/2}}\int_{\R^n}\frac{y^{2s}}{(y^2+|x-z|^2)^{(n+2s)/2}}u(z)\,dz.
\end{aligned}
\end{equation}
Moreover $U\in C^\infty(\R^n\times(0,\infty))\cap C(\R^n\times[0,\infty))$ satisfies
\begin{equation}\label{eqFractionalLaplacian:limit1}
-\lim_{y\to0^+}y^{a}U_y(x,y)=\frac{\Gamma(1-s)}{4^{s-1/2}\Gamma(s)}(-\Delta)^su(x)
\end{equation}
and
\begin{equation}\label{eqFractionalLaplacian:limit2}
-\lim_{y\to0^+}\frac{U(x,y)-U(x,0)}{y^{2s}}=\frac{\Gamma(1-s)}{4^s\Gamma(1+s)}(-\Delta)^su(x).
\end{equation}
\end{thm}

The first three formulas for $U$ in \eqref{eqFractionalLaplacian:U} are due to
\cite{Stinga, Stinga-Torrea CPDE}, while the last one was found in \cite{Caffarelli-Silvestre CPDE}.
The explicit constants appearing in the limits
\eqref{eqFractionalLaplacian:limit1} and \eqref{eqFractionalLaplacian:limit2}
were first discovered in \cite{Stinga, Stinga-Torrea CPDE}. Observe that when $s=1/2$ the first two formulas
in \eqref{eqFractionalLaplacian:U} reduce to the so-called Bochner subordination formula
see \cite{Bochner,Bochner-book}, also \cite{Jacob-Schilling,Stein-rojo}. Nevertheless,
the third formula in \eqref{eqFractionalLaplacian:U} is original from \cite{Stinga,Stinga-Torrea CPDE} even for the case $s=1/2$.

The idea for solving explicitly the extension problem by using Bessel functions was introduced in \cite{Stinga, Stinga-Torrea CPDE}.
For each $y>0$ we apply the Fourier transform in the variable $x$
to \eqref{eqFractionalLaplacian:extensionequation} to get an ODE of the form
$$\begin{cases}
f_\xi''(y)+\frac{a}{y}f_\xi'(y)=\lambda f_\xi(y)&\hbox{for}~y>0\\
f_\xi(0)=\widehat{u}(\xi)
\end{cases}$$
where $\lambda=|\xi|^2$ and $f_\xi(y)=\widehat{U}(\xi,y)$. This is a Bessel differential equation. The
condition that $U$ weakly vanishes as $y\to\infty$ translates in the fact that $f_\xi(y)\to0$ as $y\to\infty$.
Therefore the unique solution is (see \cite{Stinga, Stinga-Torrea CPDE})
\begin{equation}\label{eqFractionalLaplacian:preFourier}
\widehat{U}(\xi,y)=\frac{2^{1-s}}{\Gamma(s)}(y|\xi|)^s\K_s(y|\xi|)\widehat{u}(\xi)
\end{equation}
where $\K_s$ is the modified Bessel function of the second kind or Macdonald's function (see \cite{Lebedev}).
Notice that $\K_{1/2}(z)=(\frac{\pi}{2z})^{1/2}e^{-z}$ so, when $s=1/2$,
$\widehat{U}(\xi,y)=e^{-y|\xi|}\widehat{u}(\xi)$, which is the classical Poisson semigroup
for the harmonic extension of $u$ to the upper half space \eqref{eq:Poissonsemigroup}. By inverting the Fourier transform
in \eqref{eqFractionalLaplacian:preFourier} we obtain the functional calculus identity
$$U(x,y)=\frac{2^{1-s}}{\Gamma(s)}(y(-\Delta)^{1/2})^s\K_s(y(-\Delta)^{1/2})u(x).$$
Let us recall the following integral formula for the Bessel function (see \cite{Lebedev}):
$$\K_s(z)=\frac{1}{2}\bigg(\frac{z}{2}\bigg)^s\int_0^\infty e^{-t}e^{-z^2/(4t)}\,\frac{dt}{t^{1+s}}.$$
We choose $z=y|\xi|$ and apply the change of variables $y^2/(4t)\to t$ to get
$$\widehat{U}(\xi,y)=\frac{1}{\Gamma(s)}\int_0^\infty e^{-y^2/(4t)} e^{-t|\xi|^2}\big(|\xi|^{2s}\widehat{u}(\xi)\big)\,\frac{dt}{t^{1-s}}.$$
Because of \eqref{eq:heat semigroup pre Fourier}, this is in fact the second to last formula for $U$ in \eqref{eqFractionalLaplacian:U}.

This Bessel function analysis was applied in the extension problem for other fractional operators in
\cite{Bernardis,BDS,Caffarelli-Stinga,DeNapoli-Stinga,Nochetto,Stinga}.

The second identity in \eqref{eqFractionalLaplacian:U} follows from the first one by the change of variables $y^2/(4t)\to t$.
The third one is obtained by computing the Fourier transform of $U$ in $x$ in the first formula
$$\widehat{U}(\xi,y)=\frac{y^{2s}}{4^s\Gamma(s)}\int_0^\infty e^{-y^2/(4t)}e^{-t|\xi|^2}\widehat{u}(\xi)\,\frac{dt}{t^{1+s}}$$
and performing the change of variables $y^2/(4t|\xi|^2)\to t$.
The last convolution formula in \eqref{eqFractionalLaplacian:U},
found for the first time in \cite{Caffarelli-Silvestre CPDE}, follows
immediately from the first one, see \cite{Stinga, Stinga-Torrea CPDE} for the details.

Now that several identities for $U$ were given, all the properties established in
Theorem \ref{thmFractionalLaplacian:extensionproblem} are easy to verify.
For example, using that $\Delta e^{t\Delta}u=\partial_te^{t\Delta}u$
plus an integration by parts,
\begin{align*}
\Delta U(x,y) &= \frac{-y^{2s}}{4^s\Gamma(s)}\int_0^\infty \partial_t\bigg(\frac{e^{-y^2/(4t)}}{t^{1+s}}\bigg) e^{t\Delta}u(x)\,dt \\
&= -\tfrac{a}{y}U_y(x,y)-U_{yy}(x,y).
\end{align*}
The second identity in \eqref{eqFractionalLaplacian:U} immediately gives
that $\lim_{y\to0^+}U(x,y)=u(x)$.

Estimates for $U$ in terms of $u$ are easy to obtain by using the semigroup formulas
and the fact that, for any $y>0$,
\begin{equation}\label{eqFractionalLaplacian:uno}
\frac{y^{2s}}{4^s\Gamma(s)}\int_0^\infty e^{-y^2/(4t)}\,\frac{dt}{t^{1+s}}=1.
\end{equation}
For example, it is readily seen that $\|U(\cdot,y)\|_{L^p(\R^n)}\leq\|u\|_{L^p(\R^n)}$, for all $y\geq0$,
for any $1\leq p\leq\infty$.

By using \eqref{eqFractionalLaplacian:uno} and the semigroup
formula for the fractional Laplacian \eqref{eqFractionalLaplacian:semigroupformula},
\begin{align*}
y^aU_y(x,y) &= \frac{y^{1-2s}}{4^s\Gamma(s)}\int_0^\infty\partial_y\big(y^{2s}e^{-y^2/(4t)}\big)\big(e^{t\Delta}u(x)-u(x)\big)\,\frac{dt}{t^{1+s}} \\
&\longrightarrow\frac{2s}{4^s\Gamma(s)}\int_0^\infty\big(e^{t\Delta}u(x)-u(x)\big)\,\frac{dt}{t^{1+s}}
=-\frac{\Gamma(1-s)}{4^{s-1/2}\Gamma(s)}(-\Delta)^su(x)
\end{align*}
as $y\to0^+$.
Similarly, one can check that \eqref{eqFractionalLaplacian:limit2} holds.

At this point the reader can observe that the semigroup approach gives not only clear proofs,
but can also avoid the use of the Fourier transform and the special symmetries of the Laplacian.
Indeed, it relies only on heat semigroups and kernels.
In addition, as we have already mentioned, the methods have a wide applicability in a variety of different contexts.

The extension problem can be written in an equivalent way as an extension problem for the negative powers of the fractional
Laplacian $(-\Delta)^{-s}$. This is an immediate consequence of the third formula for $U$ in
\eqref{eqFractionalLaplacian:U} and the results of Theorem \ref{thmFractionalLaplacian:extensionproblem},
see \cite{Stinga,Stinga-Torrea CPDE}. For such explicit statement for negative powers $L^{-s}$ in other contexts
like manifolds and discrete settings, see \cite{CRSTV,DeNapoli-Stinga,Feo-Stinga-Volzone,Maldonado-Stinga}. 

\begin{thm}[Extension problem for negative powers]
Let $f\in\mathcal{S}$. The unique smooth solution $U=U(x,y):\R^n\times[0,\infty)\to\R$ to the Neumann extension problem
$$\begin{cases}
\Delta U+\frac{a}{y}U_y+U_{yy}=0&\hbox{in}~\R^n\times(0,\infty)\\
-y^aU_y(x,y)\big|_{y=0}=f(x)&\hbox{on}~\R^n
\end{cases}$$
that weakly vanishes as $y\to\infty$ is given by the formula
$$U(x,y)=\frac{1}{\Gamma(s)}\int_0^\infty e^{-y^2/(4t)}e^{t\Delta}f(x)\,\frac{dt}{t^{1-s}}.$$
Moreover,
$$\lim_{y\to0^+}U(x,y)=\frac{4^{s-1/2}\Gamma(s)}{\Gamma(1-s)}(-\Delta)^{-s}f(x).$$
\end{thm}

We consider next weak solutions to the extension problem. It is easy to
check that if $u\in\S$ then
$$[u]_{H^s(\R^n)}^2\equiv\|(-\Delta)^{s/2}u\|_{L^2(\R^n)}^2=\frac{c_{n,s}}{2}\int_{\R^n}\int_{\R^n}\frac{(u(x)-u(z))^2}{|x-z|^{n+2s}}\,dx\,dz.$$
The fractional Sobolev space $H^s(\R^n)$, $0<s<1$, is defined as the completion of $C^\infty_c(\R^n)$
under the norm
$$\|u\|_{H^s(\R^n)}^2=\|u\|^2_{L^2(\R^n)}+[u]_{H^s(\R^n)}^2.$$
Then $H^s(\R^n)$ is a Hilbert space with inner product
$$\langle u,v\rangle_{H^s(\R^n)}=\langle u,v\rangle_{L^2(\R^n)}
+\frac{c_{n,s}}{2}\int_{\R^n}\int_{\R^n}\frac{(u(x)-u(z))(v(x)-v(z))}{|x-z|^{n+2s}}\,dx\,dz.$$
The dual of $H^s(\R^n)$ is denoted by $H^{-s}(\R^n)$. The definition of the fractional Laplacian can be extended
to functions in $H^s(\R^n)$. For any $u\in H^s(\R^n)$ we define $(-\Delta)^su$ as the element on $H^{-s}(\R^n)$
that acts on $v\in H^s(\R^n)$ via
\begin{align*}
\langle(-\Delta)^su,v\rangle_{H^{-s},H^s} &= \langle (-\Delta)^{s/2}u,(-\Delta)^{s/2}v\rangle_{L^2(\R^n)} \\
&=\frac{c_{n,s}}{2}\int_{\R^n}\int_{\R^n}\frac{(u(x)-u(z))(v(x)-v(z))}{|x-z|^{n+2s}}\,dx\,dz.
\end{align*}

The extension problem can be posed in this $L^2$ setting. Notice that
$$\Delta U+\tfrac{a}{y}U_y+U_{yy}=y^{-a}\dive_{x,y}(y^a\nabla_{x,y}U).$$
The weighted Sobolev space
$$H^1_a\equiv H^1(\R^n\times(0,\infty),y^adxdy)\quad\hbox{where}~a=1-2s$$
is defined as the completion of $C^\infty_c(\R^n\times[0,\infty))$ under the norm
$$\|U\|_{H^1_a}^2=\|U\|_{L^2(\R^n\times(0,\infty),y^adxdy)}^2+\|\nabla_{x,y}U\|_{L^2(\R^n\times(0,\infty),y^adxdy)}^2.$$
Since $a\in(-1,1)$, the weight
$\omega(x,y)=y^a$ belongs to the Muckenhoupt class $A_2(\R^n\times(0,\infty))$, see \cite{Duo} for details
about these weights and \cite{Fabes,Turesson} for weighted Sobolev spaces.
Then $H^1_a$ is a Hilbert space with inner product
$$\langle U,V\rangle_{H^1_a}=\int_0^\infty\int_{\R^n}y^aUV\,dx\,dy+\int_0^\infty\int_{\R^n}y^a\nabla_{x,y}U\cdot\nabla_{x,y}V\,dx\,dy.$$
Given $u\in L^2(\R^n)$ we say that $U\in H^1_a$ is a weak solution to the extension problem \eqref{eqFractionalLaplacian:extensionequation}
if for every $V\in C^\infty_c(\R^n\times(0,\infty))$
$$\int_0^\infty\int_{\R^n}y^a\nabla_{x,y}U\cdot\nabla_{x,y}V\,dx\,dy=0$$
and $\lim_{y\to0^+}U(x,y)=u(x)$ in $L^2(\R^n)$. The proof of the following extension theorem
in weak form is just a simple verification.

\begin{thm}\label{thmFractionalLaplacian:extensionproblemweak}
Let $u\in H^s(\R^n)$. The unique weak solution $U\in H^1_a$ to the extension problem
\eqref{eqFractionalLaplacian:extensionequation} is given by \eqref{eqFractionalLaplacian:U}.
Moreover $U(\cdot,y)\in C^\infty(\R^n\times(0,\infty))\cap C([0,\infty);L^2(\R^n))$ satisfies
\eqref{eqFractionalLaplacian:limit1} and \eqref{eqFractionalLaplacian:limit2} in the sense of $H^{-s}(\R^n)$.
In addition, for any $V\in C^\infty_c(\R^n\times[0,\infty))$,
$$\int_0^\infty\int_{\R^n}y^a\nabla_{x,y}U\cdot\nabla_{x,y}V\,dx\,dy=\frac{\Gamma(1-s)}{4^{s-1/2}\Gamma(s)}
\langle (-\Delta)^su,V(\cdot,0)\rangle_{H^{-s},H^s}$$
and also
$$\mathcal{I}_s(U)\equiv\int_0^\infty\int_{\R^n}y^a|\nabla_{x,y}U|^2\,dx\,dy=\frac{\Gamma(1-s)}{4^{s-1/2}\Gamma(s)}
\int_{\R^n}|(-\Delta)^{s/2}u|^2\,dx.$$
Moreover, $U$ is the unique minimizer of the energy functional
$\mathcal{I}_s(V)$ among all functions $V\in H^1_a$ such that $\lim_{y\to0^+}V(x,y)=u(x)$ in $L^2(\R^n)$.
\end{thm}

\section{Applications of the extension problem: Harnack inequality and derivative estimates}\label{sec:Harnack}

This section is devoted to show how the extension problem can be used to prove an interior Harnack
inequality and derivative estimates for fractional harmonic functions.
These original ideas are due to Caffarelli--Silvestre \cite{Caffarelli-Silvestre CPDE}
and Caffarelli--Salsa--Silvestre \cite{Caffarelli-Salsa-Silvestre}. Obviously, such results for the fractional 
Laplacian are classical \cite{Landkof}, see also
\cite{Bass-Levin,Bogdan,Caffarelli-Silvestre-CPAM,Kassmann,SSV,Song-Vondracek} for other formulations,
techniques and nonlocal operators. Harnack inequalities using the extension technique for fractional powers of operators
in divergence form were systematically developed in \cite{Stinga-Zhang}. For applications of the extension method
to other fractional operators
see \cite{BDS,Caffarelli-Stinga,DeNapoli-Stinga,Ferrari-Franchi,Maldonado-Stinga,Roncal-Stinga,
Stinga, Stinga-Torrea CPDE, Stinga-Torrea-SIAM,Stinga-Torrea JFA,Stinga-Volzone}. Harnack inequalities
for degenerate elliptic equations like the extension equation \eqref{eqFractionalLaplacian:extensionequation}
were first proved by Fabes, Kenig and Serapioni in \cite{Fabes}.

The following reflection lemma will be needed in the proof of Harnack inequality,
see \cite{Caffarelli-Silvestre CPDE}, also \cite{Stinga-Zhang}.
In what follows $\Omega$ denotes a domain in $\R^n$ that can be unbounded.

\begin{lem}\label{lemFractionalLaplacian:reflection}
Fix $Y>0$. Suppose that a function
$U=U(x,y):\Omega\times(0,Y)\to\R$ satisfies
$$\dive(y^a\nabla_{x,y}U)=0\quad\hbox{in}~\Omega\times(0,Y)$$
in the weak sense, with
$$-\lim_{y\to0^+}y^{a}U_y(x,y)=0\quad\hbox{on}~\Omega.$$
Namely, suppose that $U$ and $\nabla_{x,y}U$ belong to $L^2(\Omega\times(0,Y),y^adxdy)$
and that for every test function $V\in C^\infty_c(\Omega\times[0,Y))$ we have
$$\int_0^\infty\int_\Omega y^a\nabla_{x,y}U\cdot\nabla_{x,y}V\,dx\,dy=0$$
and
$$-\lim_{y\to0^+}\int_\Omega y^{a}U_y(x,y)V(x,y)\,dx=0.$$
Then the even reflection of $U$ in the variable $y$ defined as
$\widetilde{U}(x,y)=U(x,|y|)$, for $y\in(-Y,Y)$, is a weak solution to
$$\dive(|y|^a\nabla_{x,y}\widetilde{U})=0\quad\hbox{in}~\Omega\times(-Y,Y).$$
\end{lem}

\begin{thm}[Interior Harnack inequality]\label{eqFractionalLaplacian:interiorHarnack}
Let $\Omega'$ be a domain such that $\Omega'\subset\subset\Omega$.
There exists a constant $c=c(\Omega,\Omega',s)>0$ such that for any solution $u\in H^s(\R^n)$ to
$$\begin{cases}
(-\Delta)^su=0&\hbox{in}~\Omega\\
u\geq0&\hbox{in}~\R^n
\end{cases}$$
we have
$$\sup_{\Omega'}u\leq c\inf_{\Omega'}u.$$
Moreover, solutions $u\in H^s(\R^n)$ to $(-\Delta)^su=0$ in $\Omega$ are locally bounded
and locally $\alpha$-H\"older continuous in $\Omega$, for some exponent
$0<\alpha<1$ that depends only on $n$ and $s$. More precisely, for any
compact set $K\subset\Omega$ there exists $C=C(c,K,\Omega)>0$ such that 
$$\|u\|_{C^{0,\alpha}(K)}\leq C\|u\|_{L^2(\R^n)}.$$
 If, in addition, $u\in L^\infty(\R^n)$ then
$$[u]_{C^{\alpha}(K)}\leq C\|u\|_{L^\infty(\R^n)}.$$
\end{thm}

\begin{proof}[Sketch of proof]
Let $U$ be the extension of $u$ given by Theorem \ref{thmFractionalLaplacian:extensionproblemweak}.
If $u\geq0$ in $\R^n$ then, by any of the formulas in \eqref{eqFractionalLaplacian:U},
$U\geq0$. Moreover, $U$ verifies the hypotheses of Lemma \ref{lemFractionalLaplacian:reflection}.
Hence the reflection $\widetilde{U}$ is a nonnegative weak solution to
$$\dive(|y|^a\nabla_{x,y}\widetilde{U})=0\quad\hbox{in}~\Omega\times(-2,2).$$
The interior Harnack inequality for divergence form degenerate elliptic equations
with $A_2$ weights applies to $\widetilde{U}$ (see \cite{Fabes}), and hence, to $u$.
The estimates follow from \eqref{eqFractionalLaplacian:U} and \eqref{eqFractionalLaplacian:uno}.
\end{proof}

The extension equation in \eqref{eqFractionalLaplacian:extensionequation} is translation
invariant in the variable $x$.
Thus, an argument based on Caffarelli's incremental quotient lemma \cite[Lemma~5.6]{Caffarelli-Cabre}
can be used to prove interior derivative estimates for fractional harmonic functions. Details
of this argument are found in \cite{Caffarelli-Salsa-Silvestre}. For similar techniques
applied to parabolic problems, see \cite{Stinga-Torrea-SIAM}.

\begin{cor}[Interior derivative estimates]\label{corFractionalLaplacian:derivativeestiamtes}
Let $u\in H^s(\R^n)\cap L^\infty(\R^n)$ be a solution to
$$(-\Delta)^su=0\quad\hbox{in}~B_2.$$
Then $u$ is smooth in the interior of $B_2$ and,
for any multi-index $\gamma\in\N_0^n$ there is a constant $C=C(|\gamma|,n,s)>0$ such that
$$\sup_{x\in B_1}|D^\gamma u(x)|\leq C\|u\|_{L^\infty(\R^n)}.$$
\end{cor}

\section{Semigroup characterization of H\"older and Zygmund spaces}\label{sec:spaces}

As we explained in the introduction, the H\"older and Schauder estimates for the fractional Laplacian can be proved
in a rather quick and elegant way by 
means of a characterization of H\"older and Zygmund spaces in terms of
heat semigroups. In this way one can avoid the use of pointwise formulas or
the Schauder estimates for the Laplacian as done by Silvestre in \cite{Silvestre Thesis, Silvestre CPAM}.
Moreover, the semigroup method allows us 
to reach the endpoint case of $\alpha+2s$ being an integer (where the appropriate regularity
spaces turn out to be different than the rather ``natural'' endpoint Lipschitz or $C^k$ spaces),
and case of $L^\infty$ right hand side.
See \cite{Caffarelli-Stinga,Silvestre Thesis, Silvestre CPAM,Stein,Zygmund} for considerations about endpoint spaces.

Let $\alpha>0$ and take \textit{any} $k\geq\lfloor\alpha/2\rfloor+1$. Define
$$\Lambda^\alpha=\Big\{u\in L^\infty(\R^n):[u]_{\Lambda^\alpha}=\sup_{x\in\R^n,t>0}|t^{k-\alpha/2}\partial_t^ke^{t\Delta}u(x)|<\infty\Big\}$$
under the norm $\|u\|_{\Lambda^\alpha}=\|u\|_{L^\infty(\R^n)}+[u]_{\Lambda^\alpha}$.
It can be seen that this definition is independent of $k$ and that the norms for different $k$ are all equivalent.

The Zygmund space $\Lambda_\ast=\Lambda_\ast^1$ is the set of functions $u\in L^\infty(\R^n)$ such that
$$[u]_{\Lambda_\ast}=\sup_{x,h\in\R^n}\frac{|u(x+h)+u(x-h)-2u(x)|}{|h|}<\infty$$
under the norm $\|u\|_{\Lambda_\ast}=\|u\|_{L^\infty(\R^n)}+[u]_{\Lambda_\ast}$, see \cite{Zygmund}.
Note that $C^{0,1}(\R^n)\subsetneq\Lambda_\ast$ continuously.
For any integer $k\geq2$ we denote
$$\Lambda_\ast^{k}=\big\{u\in C^{k-1}(\R^n):D^\gamma u\in\Lambda_\ast~\hbox{for all}~|\gamma|=k-1\big\}$$
with norm $\|u\|_{\Lambda_\ast^k}=\|u\|_{C^{k-1}(\R^n)}+\max_{|\gamma|=k-1}[D^\gamma u]_{\Lambda_\ast}$.
Then $C^{k-1,1}(\R^n)\subsetneq\Lambda_\ast^k$.

The spaces $\Lambda^\alpha$, given in terms of the rate of growth of
the heat semigroup, coincide with the classical H\"older and Zygmund spaces.
The following result for $0<\alpha<2$ can be found in \cite[Theorem~4*]{Taibleson}
and \cite{Berens-Butzer}.
For any $\alpha>0$ and when $\R^n$ is replaced
by the torus $\mathbb{T}^n$, see for example \cite{Roncal-Stinga}.
In \cite{Stein,Taibleson} a similar characterization is proved by using Poisson semigroups,
and \cite{Stinga-Torrea-SIAM} contains the case of parabolic H\"older--Zygmund spaces.
See \cite{Liu-Sjogren} for the regularity spaces associated with the Ornstein--Uhlenbeck operator,
as \cite{Marta-Torrea} for those related to the fractional powers of the parabolic harmonic oscillator.
 
\begin{thm}\label{thmFractionalLaplacian:LambdaAlpha}
If $\alpha>0$ then
$$\Lambda^\alpha=
\begin{cases}
C^{\lfloor\alpha\rfloor,\alpha-\lfloor\alpha\rfloor}(\R^n)&\hbox{if}~\alpha~\hbox{is not an integer}\\
\Lambda^k_\ast&\hbox{if}~\alpha=k~\hbox{is an integer}
\end{cases}$$
with equivalent norms.
\end{thm}

\section{H\"older and Schauder estimates with the method of semigroups}\label{sec:regularity}

This Section is devoted to show how the fractional Laplacian interacts with H\"older and Zygmund spaces.
For this we apply the method of semigroups in combination with Theorem \ref{thmFractionalLaplacian:LambdaAlpha}.
We sketch part of the (rather simple) proofs here.
In particular, the technique avoids the use of pointwise formulas, or the boundedness of Riesz transforms
on H\"older spaces, or the H\"older and Schauder estimates for the Laplacian as in \cite{Silvestre Thesis, Silvestre CPAM},
see also \cite{Stinga-Torrea JFA}.
Similar proofs but in different
contexts can be found in \cite{Caffarelli-Stinga,Marta-Torrea,Gatto-Urbina, Ma-Stinga-Torrea-Zhang, Roncal-Stinga, Stinga-Torrea-SIAM}.

The first result establishes that the fractional Laplacian $(-\Delta)^s$ behaves as an operator of order $2s$
in the scale of H\"older spaces.

\begin{thm}[H\"older estimates]
Let $u\in C^{k,\alpha}(\R^n)$, for $k\geq0$ and $0<\alpha\leq1$.
\begin{enumerate}[$(i)$]
\item If $0<2s<\alpha$ then $(-\Delta)^su\in C^{k,\alpha-2s}(\R^n)$ and
$$\|(-\Delta)^su\|_{C^{k,\alpha-2s}(\R^n)}\leq C\|u\|_{C^{k,\alpha}(\R^n)}.$$
\item If $0<\alpha<2s$ and $k\geq1$ then $(-\Delta)^su\in C^{k-1,\alpha-2s+1}(\R^n)$ and
$$\|(-\Delta)^su\|_{C^{k-1,\alpha-2s+1}(\R^n)}\leq C\|u\|_{C^{k,\alpha}(\R^n)}.$$ 
\end{enumerate}
The constants $C>0$ above depend only on $n$, $s$, $k$ and $\alpha$.
\end{thm}

The idea for $(i)$ is as follows. In view of \eqref{eqFractionalLaplacian:commute},
it is enough to prove $(i)$ for $k=0$.
By Theorem \ref{thmFractionalLaplacian:LambdaAlpha},
we only have to show that $(-\Delta)^su\in\Lambda^{\alpha-2s}$
for $u\in\Lambda^\alpha$ and $0<2s<\alpha\leq1$. For this,
by using \eqref{eqFractionalLaplacian:semigroupformula}, one can write
$$\Gamma(-s)t\partial_te^{t\Delta}[(-\Delta)^su](x)=\int_0^\infty t\partial_te^{t\Delta}\big(e^{r\Delta}u(x)-u(x)\big)\,\frac{dr}{r^{1+s}}.$$
On one hand, since $\{e^{t\Delta}\}_{t\geq0}$ is a semigroup and $u\in\Lambda^\alpha$,
\begin{align*}
\bigg|\int_0^t t\partial_te^{t\Delta}\big(e^{r\Delta}u(x)-u(x)\big)\,\frac{dr}{r^{1+s}}\bigg|
&=\bigg|\int_0^tt\partial_te^{t\Delta}\bigg[\int_0^r\partial_\rho e^{\rho\Delta}u(x)\,d\rho\bigg]\,\frac{dr}{r^{1+s}}\bigg| \\
&= t\bigg|\int_0^t\int_0^r\partial_w^2e^{w\Delta}u(x)\big|_{w=t+\rho}\,d\rho\,\frac{dr}{r^{1+s}}\bigg| \\
&\leq C[u]_{\Lambda^\alpha}t^{(\alpha-2s)/2}.
\end{align*}
On the other hand,
\begin{align*}
\bigg|\int_t^\infty t\partial_te^{t\Delta}&\big(e^{r\Delta}u(x)-u(x)\big)\,\frac{dr}{r^{1+s}}\bigg| \\
&\leq t\int_t^\infty|\partial_we^{w\Delta}u(x)\big|_{w=t+r}|\,\frac{dr}{r^{1+s}}
+|t\partial_te^{t\Delta}u(x)|\int_t^\infty\,\frac{dr}{r^{1+s}} \\
&\leq C[u]_{\Lambda^\alpha}t^{(\alpha-2s)/2}.
\end{align*}
If $k\geq1$ then Theorem \ref{thmFractionalLaplacian:LambdaAlpha}
shows that $(ii)$
is a consequence of the fact that $(-\Delta)^s:\Lambda^{k+\alpha}\to\Lambda^{k+\alpha-2s}$.
The latter can be accomplished with parallel arguments to those used to prove $(i)$.

The semigroup formula from Theorem \ref{thmFractionalLaplacian:distributionalsolvability}
and the characterization in Theorem \ref{thmFractionalLaplacian:LambdaAlpha}
permit us to prove the Schauder estimates $(-\Delta)^{-s}:C^\alpha\to C^{\alpha+2s}$
in a rather simple way. This same result, only for the case when $\alpha+2s$ is not an integer,
was obtained using pointwise formulas and the Schauder estimates for the
Laplacian in \cite{Silvestre Thesis, Silvestre CPAM}. In our case
the semigroup method permits us to include the case when $\alpha+2s\in\N$.
See \cite{Caffarelli-Stinga,Marta-Torrea,Gatto-Urbina, Ma-Stinga-Torrea-Zhang, Roncal-Stinga, Stinga-Torrea-SIAM}
for similar proofs in different contexts.

\begin{thm}[Schauder--Zygmund estimates]\label{thmFractionalLaplacian:Schaudercompactsupport}
Let $f\in C^{0,\alpha}(\R^n)$ with compact support, for some $0<\alpha\leq 1$, and define
$u$ as in Theorem \ref{thmFractionalLaplacian:distributionalsolvability}.
\begin{enumerate}[$(i)$]
\item If $\alpha+2s<1$ then $u\in C^{0,\alpha+2s}(\R^n)$ and 
$$\|u\|_{C^{0,\alpha+2s}(\R^n)}\leq C\big(\|u\|_{L^\infty(\R^n)}+\|f\|_{C^{0,\alpha}(\R^n)}\big).$$
\item If $1<\alpha+2s<2$ then $u\in C^{1,\alpha+2s-1}(\R^n)$ and 
$$\|u\|_{C^{1,\alpha+2s-1}(\R^n)}\leq C\big(\|u\|_{L^\infty(\R^n)}+\|f\|_{C^{0,\alpha}(\R^n)}\big).$$
\item If $2<\alpha+2s<3$ then $u\in C^{2,\alpha+2s-2}(\R^n)$ and 
$$\|u\|_{C^{2,\alpha+2s-2}(\R^n)}\leq C\big(\|u\|_{L^\infty(\R^n)}+\|f\|_{C^{0,\alpha}(\R^n)}\big).$$
\item If $\alpha+2s=k$, $k=1,2$, then $u\in\Lambda_\ast^k$ and
$$\|u\|_{\Lambda_\ast^k}\leq C\big(\|u\|_{L^\infty(\R^n)}+\|f\|_{C^{0,\alpha}(\R^n)}\big).$$
\end{enumerate}
The constants $C>0$ above depend only on $n$, $s$ and $\alpha$.

As a direct consequence of the solvability result in
Theorem \ref{thmFractionalLaplacian:distributionalsolvability}, $u$ is the unique
bounded classical solution to $(-\Delta)^su=f$ that vanishes at infinity.
\end{thm}

From Theorem \ref{thmFractionalLaplacian:LambdaAlpha}, the statement
of Theorem \ref{thmFractionalLaplacian:Schaudercompactsupport} reduces to
show that $(-\Delta)^{-s}:\Lambda^\alpha\to\Lambda^{\alpha+2s}$, and this is
very easy to prove. Indeed,
for any $k\geq\lfloor(\alpha+2s)/2\rfloor+1$,
\begin{align*}
|t^k\partial_t^ke^{t\Delta}[(-\Delta)^{-s}f](x)|
&= C_st^k\bigg|\int_0^\infty\partial_w^ke^{w\Delta}f(x)\big|_{w=t+r}\,\frac{dr}{r^{1-s}}\bigg| \\
& \leq C[f]_{\Lambda^\alpha}t^{(\alpha+2s)/2}.
\end{align*}

With the semigroup method we can prove the Schauder estimates in H\"older--Zygmund spaces
for the case when the right hand side is just bounded, see the details in \cite{Caffarelli-Stinga}.

\begin{thm}[Schauder--H\"older--Zygmund estimates]\label{thmFractionalLaplacian:SchauderBounded}
Let $f\in L^\infty(\R^n)$ with compact support and
define $u$ as in Theorem \ref{thmFractionalLaplacian:distributionalsolvability}.
\begin{enumerate}[$(i)$]
\item If $2s<1$ then $u\in C^{0,2s}(\R^n)$ and 
$$\|u\|_{C^{0,2s}(\R^n)}\leq C\big(\|u\|_{L^\infty(\R^n)}+\|f\|_{L^\infty(\R^n)}\big).$$
\item If $2s=1$ then $u\in\Lambda_\ast$ and 
$$\|u\|_{\Lambda_\ast}\leq C\big(\|u\|_{L^\infty(\R^n)}+\|f\|_{L^\infty(\R^n)}\big).$$
\item If $2s>1$ then $u\in C^{1,2s-1}(\R^n)$ and 
$$\|u\|_{C^{1,2s-1}(\R^n)}\leq C\big(\|u\|_{L^\infty(\R^n)}+\|f\|_{L^{\infty}(\R^n)}\big).$$
\end{enumerate}
The constants $C>0$ above depend only on $n$ and $s$.

As a direct consequence of the solvability result in
Theorem \ref{thmFractionalLaplacian:distributionalsolvability}, $u$ is the unique
bounded solution to $(-\Delta)^su(x)=f(x)$, for a.e.~$x\in\R^n$, that vanishes at infinity.
\end{thm}

For the next result, we do not assume that the right hand side has compact support.
The idea for the proof, see \cite{Silvestre Thesis,Silvestre CPAM}, also \cite{Stinga-Torrea-SIAM} for the parabolic case, is to choose 
$\eta\in C^\infty_c(B_2)$ such that $0\leq\eta\leq1$, $\eta=1$ in $B_1$, $|\nabla\eta|\leq C$ in $\R^n$,
and write $f=\eta f+(1-\eta)f=f_1+f_2$ and $u=u_1+u_2$, where $u_1$ is the solution to
$(-\Delta)^su_1=f_1$ in $\R^n$. Then, since $f_1$ has compact support,
$u_1$ is given as in Theorem \ref{thmFractionalLaplacian:distributionalsolvability}
and, therefore, Theorems \ref{thmFractionalLaplacian:Schaudercompactsupport}
and \ref{thmFractionalLaplacian:SchauderBounded} apply to it. On the other hand,
$(-\Delta)^su_2=0$ in $B_1$,
so by Corollary \ref{corFractionalLaplacian:derivativeestiamtes} we can bound, for any $k$ and $\alpha$,
$$\|u_2\|_{C^{k,\alpha}(B_{1/2})}\leq C\|u-u_1\|_{L^\infty(\R^n)}\leq C\big(\|u\|_{L^\infty(\R^n)}+\|f\|_{L^\infty(\R^n)}\big)$$
where $C>0$ depends only on $n$, $k$ and $\alpha$.
For a proof of part $(a)$ of the following result in the case when $\alpha+2s$ is not an integer
and by using pointwise formulas see \cite{Silvestre Thesis, Silvestre CPAM}.

\begin{thm}[Schauder--Zygmund estimates]
Let $u\in L^\infty(\R^n)$.
\begin{enumerate}[$(a)$]
\item Assume that $(-\Delta)^su=f\in C^{0,\alpha}(\R^n)$ for some $0<\alpha\leq1$.
Then $u$ satisfies the estimates $(i)$--$(iv)$ 
from Theorem \ref{thmFractionalLaplacian:Schaudercompactsupport}.
\item Assume that $(-\Delta)^su=f\in L^\infty(\R^n)$. Then $u$ satisfies the estimates $(i)$--$(iii)$ 
from Theorem \ref{thmFractionalLaplacian:SchauderBounded}.
\end{enumerate}
\end{thm}



\end{document}